\documentclass[10pt,a4paper]{amsart}
\usepackage{amsmath, amscd, amsfonts, amssymb, amsthm, latexsym, url, color, todonotes, bm, framed, rotating} %pdflscape}
\usepackage{graphicx}
\usepackage[left=1.2in, right=1.2in, top=1in, bottom=1.3in, includefoot, headheight=13.6pt]{geometry}
\input{xy}
\xyoption{all}

\newtheorem{theorem}{Theorem}[section]
\newtheorem{lemma}[theorem]{Lemma}

\newtheorem{proposition}[theorem]{Proposition}
\newtheorem{definition}[theorem]{Definition}
\newtheorem{conjecture}[theorem]{Conjecture}

\theoremstyle{definition}
\newtheorem{remark}[theorem]{Remark}
\newtheorem{example}[theorem]{Example}

\newcommand{\xysquare}[8]{
\[\xymatrix{
#1 \ar@{#5}[r] \ar@{#6}[d] & #2 \ar@{#7}[d]\\
#3 \ar@{#8}[r] & #4
}\]
}

\newcommand{\bb}{\mathbb}

\newcommand{\comment}[1]{}

\newcommand{\op}{\operatorname}

\renewcommand{\phi}{\varphi}

\newcommand{\roi}{\mathcal{O}}

\newcommand{\sub}[1]{{\mbox{\rm \scriptsize #1}}}

\renewcommand{\hat}{\widehat}

\renewcommand{\tilde}{\widetilde}

\newcommand{\dR}{{\mathrm{dR}}}
\renewcommand{\inf}{{\mathrm{inf}}}

\DeclareMathOperator{\Hom}{Hom}

\DeclareMathOperator{\Spec}{Spec}

\DeclareMathOperator{\Spf}{Spf}

\DeclareMathOperator{\Tor}{Tor}

\newenvironment{altenumerate}
    {\begin{list}
       {(\theenumi) }
       {\usecounter{enumi}
        \setlength{\labelwidth}{-2pt}
        \setlength{\labelsep}{2pt}
        \setlength{\leftmargin}{0pt}
        \setlength{\itemsep}{\the\smallskipamount}
        \renewcommand{\theenumi}{\roman{enumi}}
       }} {\end{list}}

\begin{document}

\title{Canonical $q$-deformations in arithmetic geometry}

\author{P.~Scholze}

\begin{abstract} In recent work with Bhatt and Morrow, we defined a new integral $p$-adic cohomology theory interpolating between \'etale and de~Rham cohomology. An unexpected feature of this cohomology is that in coordinates, it can be computed by a $q$-deformation of the de~Rham complex, which is thus canonical, at least in the derived category. In this short survey, we try to explain what we know about this phenomenon, and what can be conjectured to hold.
\end{abstract}

%\classno{}
\date{\today}
%\extraline{}

\maketitle

\tableofcontents

\setcounter{section}{0}

\section{Introduction}

Many of the most basic concepts in mathematics have so-called $q$-analogues, where $q$ is a formal variable; specializing to $q=1$ recovers the usual concept. This starts with Gau\ss's $q$-analogue
\[
[n]_q := 1+q+\ldots+q^{n-1} = \frac{q^n-1}{q-1}
\]
of an integer $n\geq 0$. One can form the $q$-factorial and $q$-binomial coefficient
\[
[n]_q! := \prod_{i=1}^n [i]_q\ ,\ {n\choose k}_q := \frac{[n]_q!}{[k]_q! [n-k]_q!}\ .
\]
One interpretation of these numbers is that if $q$ is a power of a prime, and $\bb F_q$ the corresponding finite field, then ${n\choose k}_q$ is the number of $k$-dimensional subvectorspaces of the $n$-dimensional vector space $\bb F_q^n$. As such, $q$-analogues appear naturally in counting problems over finite fields. They also appear in combinatorics, in the theory of quantum groups, in formulas for modular forms, and a whole variety of other contexts which we do not attempt to survey.

There is also a $q$-analogue of the derivative, known as the Jackson $q$-derivative, \cite{jackson}. The $q$-derivative of the function $x\mapsto x^n$ is not $nx^{n-1}$, but $[n]_q x^{n-1}$. In general, the $q$-derivative of $f(x)$ is given by
\[
(\nabla_q f)(x) = \frac{f(qx)-f(x)}{qx-x}\ .
\]
Taking the limit as $q\to 1$ recovers the usual derivative. Thus, the $q$-derivative is a finite difference quotient, and the theory of $q$-derivatives is tied closely with the theory of difference equations.

However, one important property is lost when taking a $q$-analogue: Invariance under coordinate transformations. Namely, the formula for the $q$-derivative shows that its value at $x$ depends on the values of $f$ at the two different points $x$ and $qx$ (which may be quite far from $x$). Perhaps more to the point, we do not see any possible $q$-analogue of the chain rule.

The theme of this paper is that in arithmetic situations, a shadow of this $q$-deformation should be independent of coordinates, which seems to be a new phenomenon. Throughout the paper, we have to assume that the variable $q$ is infinitesimally close to $1$; technically, we take the $(q-1)$-adic completion of all intervening objects. It is a very interesting question whether this is really necessary; all the explicit coordinate-independent formulas that we can write down do not involve the $(q-1)$-adic completion. This may, however, be a reflection of our inability to make any nontrivial coordinate-independent computation.

At first sight, assuming $q$ infinitesimally close to $1$ seems to be almost equivalent to taking the limit as $q\to 1$. Indeed, if $f(x)\in \bb R[x]$ is a real polynomial, then there is a formula for the $q$-derivative as a Taylor series expansion in $q-1$ in terms of the higher derivatives of $f$. However, this formula involves denominators, and works only over $\bb Q$-algebras. Thus, the problem becomes potentially interesting if one works over $\bb Z$, which we will do in the paper.

The basic conjecture is the following:

\begin{conjecture}\label{conjintro} Let $R$ be a smooth $\bb Z$-algebra. There is a functorial $q$-de~Rham complex $q\!\op-\!\Omega_R$ of $\bb Z[[q-1]]$-modules which can be described explicitly (in the derived category) once one fixes coordinates, i.e.~an \'etale map $\bb Z[T_1,\ldots,T_d]\to R$. For example, if $R=\bb Z[x]$, then $q\!\op-\!\Omega_R$ is computed by
\[
\bb Z[x][[q-1]]\buildrel{\nabla_q}\over\longrightarrow \bb Z[x][[q-1]]\ .
\]
\end{conjecture}

The $q$-de~Rham complex was first studied by Aomoto, \cite{aomoto}, and its cohomology is sometimes called Aomoto--Jackson cohomology.

We refer to Section~\ref{sec:qdeRham} for more details on Conjecture~\ref{conjintro}. It is important to note that we do not expect the complex itself to be independent of coordinates; it should only be independent up to (canonical) quasi-isomorphism. Even in this simplest example $R=\bb Z[x]$, we are unable to write down explicitly the expected quasi-isomorphism between the $q$-de~Rham complexes
\[
\bb Z[x][[q-1]]\buildrel{\nabla_q}\over\longrightarrow \bb Z[x][[q-1]]
\]
and
\[
\bb Z[y][[q-1]]\buildrel{\nabla_q}\over\longrightarrow \bb Z[y][[q-1]]
\]
for an isomorphism $\bb Z[x]\cong \bb Z[y]$, like $y=x+1$.

Assuming the conjecture, one can define $q$-de~Rham cohomology groups
\[
H^i_{q\!\op-\!\dR}(X)
\]
for any smooth $\bb Z$-scheme $X$, as the hypercohomology of $q\!\op-\!\Omega_X$. We expect that these interpolate between de~Rham cohomology (for $q=1$) and singular cohomology (for $q\neq 1$), if $X$ is proper over $\bb Z[1/N]$ for some integer $N$; again, we refer to Section~\ref{sec:qdeRham} for details. The occurence of singular cohomology in this context may be unexpected, and reflects deep comparison isomorphisms from $p$-adic Hodge theory. We note that this degeneration from singular cohomology to de~Rham cohomology implies formally (by semicontinuity) that torsion in de~Rham cohomology is at least as large as in singular cohomology, cf.~Theorem~\ref{thmtorsion}, which is proved unconditionally in~\cite{BMS}. This ``explains'' certain pathologies of algebraic geometry in positive characteristic, such as the fact that Enriques surfaces in characteristic $2$ have nonzero $H^1_\dR$ (contrary to the situation in any other characteristic): This is forced by their fundamental group being $\bb Z/2\bb Z$ over the complex numbers.

The independence of $q\!\op-\!\Omega_R$ only up to quasi-isomorphism makes the conjecture (much) more ambiguous than we would like it to be. Unfortunately, we do not know a good way of making it more precise without obscuring the idea.\footnote{We note in particular that if the conjecture is correct, it could be ``correct in more than one way'', meaning that there might be several ways of ``making the $q$-de~Rham complex independent of coordinates''. This is because there might be several possible choices for the coordinate transformations. We actually expect this to happen in the relative situation, cf.~Example~\ref{ex:variants}~(ii) below.} Let us mention however that there is another example of a similar sort:

\begin{example} Let $R$ be a smooth $\bb F_p$-algebra. Then there is a functorial de~Rham--Witt complex $W\Omega^\bullet_{R/\bb F_p}$, cf.~\cite{IllusiedRWitt}. If one fixes an \'etale map $\bb F_p[T_1,\ldots,T_d]\to R$, then there is a unique deformation of $R$ to an \'etale map $\bb Z/p^n\bb Z[T_1,\ldots,T_d]\to R_n$ for all $n\geq 1$; let $\tilde{R} = \varprojlim_n R_n$, which is the $p$-adic completion of a smooth $\bb Z_p$-algebra. Then $W\Omega^\bullet_{R/\bb F_p}$ is computed in the derived category by the ($p$-adically completed) de~Rham complex of $\tilde{R}/\bb Z_p$, more precisely by
\[
\varprojlim_n \Omega^\bullet_{R_n/(\bb Z/p^n\bb Z)}\ .
\]
The first proof that the de~Rham complex of a lift to characteristic zero is independent of the lift was through Grothendieck's formalism of the crystalline site, cf.~\cite{BerthelotOgus}.\footnote{We challenge the reader to make the familiar statement ``the de~Rham complex of a lift is independent of the choice of a lift'' into a precise and useful statement (say, good enough to define crystalline cohomology globally, remembering that one cannot glue in the derived category) without saying how it is independent of the lift (i.e., proving it), or using $\infty$-categorical language.}
\end{example}

The $q$-de~Rham cohomology of this paper can be regarded as an analogue of crystalline cohomology if one replaces the base ring $\bb F_p$ by the integers $\bb Z$. Accordingly, the deformation $\bb Z_p$ of $\bb F_p$ is replaced by the deformation $\bb Z[[q-1]]$ of $\bb Z$.

The motivation for the conjectures in this paper comes from joint work with Bhargav Bhatt and Matthew Morrow, \cite{BMS}, that proves these conjectures after base change from $\bb Z[[q-1]]$ to Fontaine's ring $A_\inf$ (constructed from $\bb C_p$), for any prime $p$.\footnote{We are hopeful that the methods from \cite{BMS} can be extended to prove the conjectures in this paper. Unfortunately, this would not produce an analogue of the crystalline site or the de~Rham--Witt complex.} The paper \cite{BMS} relies heavily on the approach to $p$-adic Hodge theory built on the theory of perfectoid spaces, and in particular the almost purity theorem, as developed by Faltings, \cite{FaltingsJAMS}, \cite{FaltingsAlmostEtale}, and extended by the author in \cite{scholzethesis}, \cite{scholzepadichodge}, cf.~also related work of Kedlaya--Liu, \cite{KedlayaLiu}. We hope that this survey can convey some of the fascination originating from the theory of perfectoid spaces without actually making perfectoid spaces (and the surrounding technical baggage) appear on stage.

Let us remark that the theory of diamonds, \cite{ScholzeLectureNotes}, makes precise the idea (going back at least to Faltings, \cite{FaltingsKodaira}) that $\Spf A_\inf$ is a ``model'' of $\Spf \bb Z_p\times_{\Spec \bb F_1} \Spf \roi_{\bb C_p}$, where $\roi_{\bb C_p}$ is the ring of integers of $\bb C_p$. This vaguely suggests that the ring $\Spf \bb Z[[q-1]]$ occuring globally is related to the completion of the unknown $\Spec \bb Z\times_{\Spec \bb F_1} \Spec \bb Z$ along the diagonal. Let us also add that $q$-de~Rham cohomology admits many semilinear operations (as detailed in Section~\ref{sec:operations}) which are reminiscent of a shtuka structure; indeed, after base change to $A_\inf$, one gets mixed-characteristic shtukas in the sense of \cite{ScholzeLectureNotes} (which in this situation were first defined by Fargues, \cite{FarguesBK}).

Unfortunately, in the present context, we do not know an analogue of either the de~Rham--Witt complex or the crystalline site. Similarly, we do not know what it would mean to deform a smooth $\bb Z$-algebra $R$ over the completion of $\Spec \bb Z\times_{\Spec \bb F_1} \Spec \bb Z$ along the diagonal.

Other evidence for these conjectures comes from known structures appearing in abstract $p$-adic Hodge theory such as the theory of Wach modules, \cite{Wach}, \cite{BergerWach}.

In Section~\ref{sec:cohom}, we recall basic facts about singular and de~Rham cohomology of arithmetic varieties. In Section~\ref{sec:qdeRham}, we formulate the conjectures about the existence of $q$-de~Rham cohomology. Next, in Section~\ref{sec:known}, we explain what is known about these conjectures from \cite{BMS}, and sketch the proof in Section~\ref{sec:ideas}. In Section~\ref{sec:operations}, we formulate further conjectures about the structure of the $q$-de~Rham complex, and in particular conjecture the existence of many semilinear operators on it, inducing interesting semilinear operations on $q$-de~Rham cohomology. The relative situation is discussed in Section~\ref{sec:qconn}; we expect a canonical $q$-deformation of the Gau\ss--Manin connection. In the final Sections~\ref{sec:examples} and~\ref{sec:variants}, we discuss some examples and variants.

{\bf Acknowledgments.} This paper was written in relation to the Fermat Prize awarded by the Universit\'e Paul Sabatier in Toulouse. We are very thankful for this opportunity to express these (sometimes vague) ideas. The ideas expressed here were formed in discussions with Bhargav Bhatt and Matthew Morrow, and the author wants to thank them heartily. Moreover, he wants to thank Laurent Fargues and Kiran Kedlaya for useful discussions; Kedlaya had ideas about a $\bb Z[[q-1]]$-valued cohomology theory before, cf.~\cite{kedlayaglobalhodge}. This work was done while the author was a Clay Research Fellow.

\section{Cohomology of algebraic varieties}\label{sec:cohom}

This paper deals with the different cohomology theories associated with arithmetic varieties, and the comparison theorems.

To get started, recall the classical comparison between singular cohomology and de~Rham cohomology over the complex numbers. For this, let $X$ be a complex manifold of complex dimension $d$, thus real dimension $2d$. Regarding $X$ as a topological space, we have the singular cohomology groups $H^i(X,\bb Z)$, which vanish outside the range $0\leq i\leq 2d$. On the other hand, we can build the holomorphic de~Rham complex
\[
\Omega^\bullet_{X/\bb C} = \mathcal{O}_X\buildrel\nabla\over\to \Omega^1_{X/\bb C}\buildrel\nabla\over\to \Omega^2_{X/\bb C}\to\ldots\to \Omega^d_{X/\bb C}\to 0\ .
\]
Here, $\mathcal{O}_X$ denotes the sheaf of holomorphic functions on $X$, $\Omega^1_{X/\bb C}$ is the sheaf of holomorphic K\"ahler differentials (which is locally free of rank $d$ over $\mathcal{O}_X$), and $\Omega^i_{X/\bb C}$ is the $i$-th exterior power of $\Omega^1_{X/\bb C}$ over $\mathcal{O}_X$. The hypercohomology groups $H^i_\dR(X)$ of $\Omega_{X/\bb C}^\bullet$ are called de~Rham cohomology; these are complex vector spaces, which vanish outside the range $0\leq i\leq 2d$. Now one has the fundamental comparison result:

\begin{theorem} There is a canonical isomorphism $H^i_\dR(X)\cong H^i(X,\bb Z)\otimes_{\bb Z} \bb C$.
\end{theorem}

More precisely, embedding the constant $\bb C$ into $\mathcal{O}_X$ induces a map $\bb C\to \Omega^\bullet_{X/\bb C}$. The Poincar\'e lemma states that this is an isomorphism of complexes of sheaves; more precisely, one can cover $X$ by open balls, and for any open ball $U\subset X$, the complex
\[
0\to \bb C\to \mathcal{O}_X(U)\to \Omega^1_{X/\bb C}(U)\to \ldots\to \Omega^d_{X/\bb C}(U)\to 0
\]
is exact: Any closed differential form can be integrated on an open ball.

The simplest example is the case $X=\bb C^\times$, in which case $H^1(X,\bb Z) = \Hom(H_1(X,\bb Z),\bb Z)=\bb Z$ is generated by the class $\alpha\in H^1(X,\bb Z)$ dual to the loop $\gamma: [0,1]\to X\ : t\mapsto e^{2\pi i t}$ running once around the puncture. On the other hand, as $X$ is Stein, $H^i_\dR(X)$ is computed by the complex of global sections
\[
\mathcal{O}_X(X)\to \Omega^1_{X/\bb C}(X)\ .
\]
Any holomorphic $1$-form on $X$ can be written uniquely as $\sum_{n\in \bb Z} a_n z^n dz$ for coefficients $a_n\in \bb C$ subject to some convergence conditions. One can integrate this form as long as $a_{-1}=0$; this shows that $H^1_\dR(X) = \bb C$, generated by the class of $\omega = \frac{dz}{z}$. Concretely, the isomorphism $H^1(X,\bb Z)\otimes_{\bb Z} \bb C\cong H^1_\dR(X)$ is given in this case by sending $\alpha$ to
\[
\int_\gamma \omega = \int_0^1 e^{-2\pi i t} d(e^{2\pi i t}) = 2\pi i \int_0^1 dt = 2\pi i\ .
\]

In this paper, we are interested in the case where $X$ is given by the vanishing locus of polynomial equations with $\bb Z$-coefficients. More precisely, from now on let $X$ be a smooth, separated scheme of finite type over $\bb Z$. Thus, locally $X$ is of the form
\[
\Spec \bb Z[T_1,\ldots,T_n] / (f_1,\ldots,f_m)\ ,
\]
for some functions $f_1,\ldots,f_m\in \bb Z[T_1,\ldots,T_n]$, $m\leq n$, subject to the Jacobian criterion, i.e.~the ideal generated by all maximal minors of the matrix $(\frac{\partial f_i}{\partial T_j})_{i,j}$ is the unit ideal. This gives rise to a complex manifold $X(\bb C)$, locally given by the set of complex solutions to $f_1(T_1,\ldots,T_n) = \ldots = f_m(T_1,\ldots,T_n) = 0$.

In this situation, we can consider the algebraic de~Rham complex
\[
\Omega^\bullet_{X/\bb Z} = \mathcal{O}_X\buildrel\nabla\over\to \Omega^1_{X/\bb Z}\buildrel\nabla\over\to \Omega^2_{X/\bb Z}\to\ldots\to \Omega^d_{X/\bb Z}\to 0\ ,
\]
and we denote by $H^i_\dR(X)$ its hypercohomology groups. These groups form an integral structure for the de~Rham cohomology groups of $X(\bb C)$:

\begin{theorem}[Grothendieck] There is a canonical isomorphism $H^i_\dR(X)\otimes_{\bb Z} \bb C\cong H^i_\dR(X(\bb C))$.
\end{theorem}

As an example, take $X=\bb G_m = \Spec \bb Z[T^{\pm 1}]$, where $\bb Z[T^{\pm 1}] = \bb Z[T,U]/(TU-1)$ denotes the ring of Laurent polynomials, so that $X(\bb C) = \bb C^\ast$. In this case, the cohomology groups $H^i_\dR(X)$ are computed by the complex
\[
\bb Z[T^{\pm 1}]\to \bb Z[T^{\pm 1}] dT\ : \sum_{n\in \bb Z} a_n T^n\mapsto \sum_{n\in \bb Z} na_n T^{n-1} dT\ .
\]
In particular, $\omega=\frac{dT}T$ defines a class in $H^1_\dR(X)$. As a consequence, we see that in the chain of isomorphisms
\[
H^i_\dR(X)\otimes_{\bb Z}\bb C\cong H^i_\dR(X(\bb C))\cong H^i(X(\bb C),\bb Z)\otimes_{\bb Z} \bb C\ ,
\]
the integral structures $H^i_\dR(X)$ and $H^i(X(\bb C),\bb Z)$ are not identified: In this example, they differ by a factor of $2\pi i$. In general, the entries of the transition matrix are known as ``periods'', which form a very interesting class of (often) transcendental numbers; we refer to Kontsevich--Zagier, \cite{KontsevichZagier}, for a precise discussion.

In the example, we observe however that $H^1_\dR(X)$ has in addition many torsion classes. Namely, for any $n\in \bb Z$, $n\neq -1$, the function $T^n dT$ integrates to $\frac 1{n+1} T^{n+1}$; in other words, we cannot integrate it over $\bb Z$, but only $(n+1) T^n dT$ can be integrated over $\bb Z$. This shows that in fact
\[
H^1_\dR(X) = \bb Z\oplus \bigoplus_{n\in \bb Z, n\neq -1} \bb Z/(n+1)\bb Z\ .\footnote{A related observation is that $H^1_\dR(\bb A^1_{\bb Z})=\bigoplus_{n\geq 0} \bb Z/(n+1)\bb Z\neq 0$, so that over $\bb Z$, de~Rham cohomology is not $\bb A^1$-invariant, and thus does not qualify as ``motivic'' in the sense of Voevodsky, \cite{voevodsky}. Similarly, the $q$-de~Rham cohomology studied in this paper will not be $\bb A^1$-invariant; in some sense, the failure will now happen at nontrivial roots of unity, and in particular even in characteristic $0$.}
\]
In order to avoid this ``pathology'', we will often assume that $X$ is in addition proper over $\bb Z[\frac 1N]$ for some integer $N\geq 1$. In this case, $X(\bb C)$ is a compact complex manifold, and for all $i\geq 0$, $H^i_\dR(X)$ is a finitely generated $\bb Z[\frac 1N]$-module.

Note that this excludes the example $X=\bb G_m$. However, one can recover $H^1(\bb G_m)$ equivalently in $H^2(\bb P^1)$. More precisely, let $X=\bb P^1_{\bb Z}$ be the projective line over $\bb Z$. This is covered by two copies of the affine line $\mathbb A^1$, glued along $\bb G_m$. As $\mathbb A^1(\bb C) = \bb C$ is contractible, the Mayer--Vietoris sequence gives an isomorphism $H^1(\bb G_m(\bb C),\bb Z)\cong H^2(\bb P^1(\bb C),\bb Z)$. On the other hand, one computes $H^2_\dR(\bb P^1) = \bb Z$ (without extra torsion), generated by the image of $\frac{dT}T$ under the boundary map $H^1_\dR(\bb G_m)\to H^2_\dR(\bb P^1)$.

In many examples, one uses the isomorphism between singular and de~Rham cohomology as a tool to understand singular cohomology: While singular cohomology is quite abstract, differential forms are amenable to computation. One example close to the interests of the author is the cohomology of arithmetic groups: These can be rewritten as singular cohomology groups of locally symmetric varieties. Under the isomorphism with de~Rham cohomology, one gets a relation to automorphic forms, which are a very powerful tool.

Note however that the comparison isomorphism forgets about all torsion classes present in $H^i(X(\bb C),\bb Z)$. One classical situation is the case where $X$ is an Enriques surface, so that $X$ has a double cover by a K3 surface. This shows that $\pi_1(X(\bb C))= \bb Z/2\bb Z$, which implies that $H^2(X(\bb C),\bb Z)$ has a nontrivial $2$-torsion class. Interestingly, in this case it turns out that if $X$ is proper over $\bb Z[\frac 1N]$, where $2$ does not divide $N$, then $H^2_\dR(X)$ has a nontrivial $2$-torsion class as well; equivalently, $H^1_\dR(X_{\mathbb F_2})\neq 0$, where $X_{\mathbb F_2}$ denotes the fiber of $X$ over $\Spec \mathbb F_2\to \Spec \bb Z[\frac 1N]$. This is a well--known ``pathology'' of Enriques surfaces in characteristic $2$, cf.~\cite[Proposition 7.3.5]{IllusiedRWitt}.

One goal of the paper \cite{BMS} with Bhatt and Morrow was to show that this is in fact a special case of a completely general phenomenon. For any abelian group $A$ and integer $n\geq 1$, let $A[n]\subset A$ be the kernel of multiplication by $n$.

\begin{theorem}\label{thmtorsion} Assume that $X$ is proper and smooth over $\bb Z[\frac 1N]$. For all $i\geq 0$ and $n\geq 1$ coprime to $N$, the order of $H^i_\dR(X)[n]$ is at least the order of $H^i(X(\bb C),\bb Z)[n]$.
\end{theorem}

In other words, torsion in singular cohomology forces torsion in de~Rham cohomology. Note that in the theorem one can assume that $n$ is a power of a prime $p$, in which case the statement depends only on the base change of $X$ to $\bb Z_p$ (if one replaces singular cohomology by \'etale cohomology). We prove the result more generally for any $X$ over a possibly highly ramified extension of $\bb Z_p$; more precisely, we prove it for any proper smooth formal scheme over the ring of integers $\mathcal{O}_{\bb C_p}$ in the completed algebraic closure $\bb C_p$ of $\bb Q_p=\bb Z_p[\frac 1p]$.

The theorem gives a means for studying torsion in singular cohomology using differential forms.

\section{$q$--de~Rham cohomology}\label{sec:qdeRham}

We deduce Theorem~\ref{thmtorsion} from the construction of a new integral $p$-adic cohomology theory interpolating between \'etale and de~Rham cohomology, which we will discuss further in Section~\ref{sec:known}. Starting with a smooth scheme over $\bb Z$, we expect a certain refinement of our cohomology theory, as follows.

Let $\bb Z[[q-1]]$ denote the $(q-1)$-adic completion of $\bb Z[q]$. In other words, $\bb Z[[q-1]]$ consists of all power series $\sum_{n\geq 0} a_n (q-1)^n$ with coefficients $a_n\in \bb Z$. For any integer $n\geq 0$, we have Gau\ss's $q$-analogue
\[
[n]_q = \frac{q^n-1}{q-1} = 1 + q + ... + q^{n-1}
\]
of the integer $n$. Setting $q=1$ recovers $[n]_1 = n$.

In the example $R=\bb Z[T]$, we can introduce a $q$-deformation $q\!\op-\!\Omega^\bullet_{R[[q-1]]/\bb Z[[q-1]]}$ of the de Rham complex $\Omega^\bullet_{R/\bb Z}$ as follows:
\[
q\!\op-\!\Omega^\bullet_{R[[q-1]]/\bb Z[[q-1]]} = R[[q-1]]\buildrel\nabla_q\over\to \Omega^1_{R/\bb Z}[[q-1]]: T^n\mapsto [n]_q T^{n-1} dT\ .
\]
This is a complex of $\bb Z[[q-1]]$-modules such that specializing at $q=1$ recovers $\Omega^\bullet_{R/\bb Z}$:
\[
q\!\op-\!\Omega^\bullet_{R[[q-1]]/\bb Z[[q-1]]}\otimes_{\bb Z[[q-1]]} \bb Z = \Omega^\bullet_{R/\bb Z}\ .
\]
Let us give a better formula for the differential $\nabla_q: R[[q-1]]\to \Omega^1_{R/\bb Z}[[q-1]]\cong R[[q-1]] dT$. For any function $f(T)\in R[[q-1]]$,
\[
\nabla_q(f(T)) = \frac{f(qT)-f(T)}{qT-T} dT\ .
\]
Indeed, applying this to $f(T)=T^n$ gives
\[
\nabla_q(T^n) = \frac{q^nT^n - T^n}{qT-T} dT = \frac{q^n-1}{q-1}T^{n-1} dT = [n]_q T^{n-1} dT\ .
\]
In other words, $\nabla_q$ is a finite $q$-difference quotient.

Note that so far we could have worked over $\bb Z[q]\subset \bb Z[[q-1]]$; the restriction to power series will only become important later. In this setting (with fixed coordinates), $\nabla_q$ is also known as the Jackson derivative, \cite{jackson}, and the resulting $q$--de~Rham cohomology as Aomoto--Jackson cohomology, \cite{aomoto}; there is an extensive literature on the subject which (due to ignorance of the author) we do not try to survey. Let us just compute that
\[
H^1(q\!\op-\!\Omega^\bullet_{R[[q-1]]/\bb Z[[q-1]]}) = \bb Z[[q-1]]\oplus \widehat{\bigoplus}_{n\in \bb Z,n\neq -1} \bb Z[[q-1]] / [n+1]_q \bb Z[[q-1]]\ ;
\]
here the direct sum is $(q-1)$-adically completed (technically, one has to take the derived $(q-1)$-adic completion). Thus, where we previously had $n$-torsion, we now get $[n]_q$-torsion (i.e., at nontrivial $n$-th roots of unity), which spreads out into characteristic $0$.

Now we try to extend this definition to all smooth $R$-algebras. Locally on $\Spec R$, any smooth $R$-algebra admits an \'etale map $\square: \bb Z[T_1,\ldots,T_d]\to R$, which we will refer to as a framing. This induces a (formally) \'etale map $\bb Z[T_1,\ldots,T_d][[q-1]]\to R[[q-1]]$. In the example, to define the $q$--derivative, we had to make sense of the automorphism $f(T)\mapsto f(qT)$. In coordinates, for any $i=1,\ldots,d$, we have the automorphism $\gamma_i$ of $\bb Z[T_1,\ldots,T_d][[q-1]]$ sending $T_i$ to $qT_i$ and $T_j$ to $T_j$ for $j\neq i$. As $\bb Z[T_1,\ldots,T_d][[q-1]]\to R[[q-1]]$ is formally \'etale, these automorphisms lift uniquely to automorphisms $\gamma_i$ of $R[[q-1]]$. Here, we crucially use the $(q-1)$-adic completion.

Using these notations, we can define the $q$--de~Rham complex
\[
q\!\op-\!\Omega^\bullet_{R[[q-1]]/\bb Z[[q-1]],\square} := R[[q-1]]\buildrel\nabla_q\over\to \Omega^1_{R/\bb Z}[[q-1]]\to\ldots\to \Omega^d_{R/\bb Z}[[q-1]]\ ,
\]
where
\[
\nabla_q(f) = \sum_{i=1}^d \frac{\gamma_i(f)-f}{qT_i-T_i} dT_i\ ,
\]
and the higher differentials are defined similarly. Again, $q\!\op-\!\Omega^\bullet_{R[[q-1]]/\bb Z[[q-1]],\square}$ is a complex of $\bb Z[[q-1]]$-modules, and
\[
q\!\op-\!\Omega^\bullet_{R[[q-1]]/\bb Z[[q-1]],\square}\otimes_{\bb Z[[q-1]]} \bb Z = \Omega^\bullet_{R/\bb Z}\ .
\]

In examples, one observes quickly that $q\!\op-\!\Omega^\bullet_{R[[q-1]]/\bb Z[[q-1]],\square}$ depends on the choice of coordinates. However, its cohomology groups should not depend on this choice:

\begin{conjecture}\label{ConjA} The complex $q\!\op-\!\Omega^\bullet_{R[[q-1]]/\bb Z[[q-1]],\square}$ is independent of the choice of coordinates up to canonical quasi-isomorphism. More precisely, there is a functor $R\mapsto q\!\op-\!\Omega_R$ from the category of smooth $\bb Z$-algebras to the $\infty$-category of $E_\infty$-$\bb Z[[q-1]]$-algebras, such that $q\!\op-\!\Omega_R$ is computed by $q\!\op-\!\Omega^\bullet_{R[[q-1]]/\bb Z[[q-1]],\square}$ for any choice of framing $\square$.
\end{conjecture}

\begin{remark} The notion of an $E_\infty$-algebra is a weakening of the notion of a commutative differential graded algebra, where commutativity only holds ``up to coherent higher homotopy'', cf.~e.g.~\cite{LurieHigherAlgebra}. The $q$-de~Rham complex is not a commutative differential graded algebra, as seen by the asymmetry in the $q$-Leibniz rule
\[
\nabla_q(f(T)g(T)) = g(T) \nabla_q(f(T)) + f(qT) \nabla_q(g(T))\ .
\]
However, it can be shown that $q\!\op-\!\Omega^\bullet_{R[[q-1]]/\bb Z[[q-1]],\square}$ is naturally an $E_\infty$-$\bb Z[[q-1]]$-algebra. A related phenomenon is that the category of modules with $q$-connections is symmetric monoidal, although it is a nontrivial exercise to write down the tensor product, cf.~Section~\ref{sec:qconn} below for further discussion.
\end{remark}

If the conjecture is true, then we can glue the complexes $q\!\op-\!\Omega_R$ to get a deformation $q\!\op-\!\Omega_X$ of the de~Rham complex for any smooth scheme $X$ over $\bb Z$. One can then define $q$-de~Rham cohomology groups
\[
H^i_{q\!\op-\!\dR}(X) := \bb H^i(X,q\!\op-\!\Omega_X)
\]
as the hypercohomology groups of $q\!\op-\!\Omega_X$. Specializing at $q=1$ gives de~Rham cohomology; more precisely, taking into account the $\Tor_1$-term, we have short exact sequences
\[
0\to H^i_{q\!\op-\!\dR}(X)/(q-1)\to H^i_\dR(X)\to H^{i+1}_{q\!\op-\!\dR}(X)[q-1]\to 0\ ,
\]
where the last term denotes the $(q-1)$-torsion. If $X$ is proper over $\bb Z[\frac 1N]$ for some $N$, one can deduce from this (and $(q-1)$-adic completeness) that $H^i_{q\!\op-\!\dR}(X)$ is a finitely generated $\bb Z[\frac 1N][[q-1]]$-module for all $i\in \bb Z$.

On the other hand, after inverting $q-1$, we expect a relation to singular cohomology.

\begin{conjecture}\label{ConjB} Assume that $X$ is proper and smooth over $\bb Z[\frac 1N]$. There are isomorphisms
\[
H^i_{q\!\op-\!\dR}(X)[\tfrac 1{q-1}]\cong H^i(X(\bb C),\bb Z)\otimes_{\bb Z} \bb Z[\tfrac 1N]((q-1))\ .
\]
\end{conjecture}

We do not expect these isomorphisms to be canonical; rather, we expect canonical isomorphisms after base extension to Fontaine's period ring $A_\mathrm{inf}$ for any prime $p$ not dividing $N$; we will discuss this in Section~\ref{sec:known} along with the results of \cite{BMS}. However, taking together these comparison isomorphisms implies, by the structure result for modules over principal ideal domains\footnote{Note that $\bb Z[\tfrac 1N]((q-1))$ is a principal ideal domain!}, that an isomorphism as in the conjecture exists.

We note that Conjecture~\ref{ConjA} and Conjecture~\ref{ConjB} together imply Theorem~\ref{thmtorsion} by a standard semicontinuity argument.

As observed earlier, the $q$-de~Rham complex tends to be interesting at roots of unity. Let us give a description of the cohomology groups of $q\!\op-\!\Omega_R$ after specialization at a $p$-th root of unity. Let $\Phi_p(q) = [p]_q = \frac{q^p-1}{q-1}$ be the $p$-th cyclotomic polynomial. Note that by sending $q$ to $\zeta_p$, $\bb Z[[q-1]]/\Phi_p(q) = \bb Z_p[\zeta_p]$, where $\zeta_p$ is a primitive $p$-th root of unity. Fix any framing $\square: \bb Z[T_1,\ldots,T_d]\to R$. Let $\hat{R}$ be the $p$-adic completion of $R$. Note that there is a unique lift of Frobenius $\varphi: \hat{R}\to \hat{R}$ sending $T_i$ to $T_i^p$; again, this follows from \'etaleness of $\square$. There is an identification
\[
R[[q-1]]/\Phi_p(q) = \hat{R}[\zeta_p]\ ,
\]
and we extend $\varphi$ to $\hat{R}[\zeta_p]$ by $\varphi(\zeta_p)=\zeta_p$.

\begin{proposition}\label{propcartier} Let $R$ be a smooth $\bb Z$-algebra with framing $\square: \bb Z[T_1,\ldots,T_d]\to R$. For brevity, write $q\!\op-\!\Omega_R^\square := q\!\op-\!\Omega^\bullet_{R[[q-1]]/\bb Z[[q-1]],\square}$.
\begin{altenumerate}
\item[{\rm (i)}] The image of the map $\varphi: \hat{R}[\zeta_p]\to \hat{R}[\zeta_p] = R[[q-1]]/\Phi_p(q)$ lands in the kernel of
\[
\nabla_q: R[[q-1]]/\Phi_p(q)\to \Omega^1_{R/\bb Z}[[q-1]]/\Phi_p(q)\ .
\]
This induces an isomorphism
\[
\hat{R}[\zeta_p] = H^0(q\!\op-\!\Omega_R^\square/\Phi_p(q))\ .
\]
\item[{\rm (ii)}] The boundary map
\[
\partial: H^0(q\!\op-\!\Omega_R^\square/\Phi_p(q))\to H^1(q\!\op-\!\Omega_R^\square/\Phi_p(q))
\]
associated with the short exact sequence of complexes
\[
0\to q\!\op-\!\Omega_R^\square/\Phi_p(q)\buildrel{\Phi_p(q)}\over\longrightarrow q\!\op-\!\Omega_R^\square/\Phi_p(q)^2\to q\!\op-\!\Omega_R^\square/\Phi_p(q)\to 0\ ,
\]
is a continuous $\bb Z_p[\zeta_p]$-linear derivation of $\hat{R}[\zeta_p] = H^0(q\!\op-\!\Omega_R^\square/\Phi_p(q))$. The induced map
\[
\Omega^1_{R/\bb Z}\otimes_R \hat{R}[\zeta_p]\to H^1(q\!\op-\!\Omega_R^\square/\Phi_p(q))
\]
is an isomorphism.
\item[{\rm (iii)}] For all $i\geq 0$, cup product induces an isomorphism
\[
\Omega^i_{R/\bb Z}\otimes_R \hat{R}[\zeta_p]\cong H^i(q\!\op-\!\Omega_R^\square/\Phi_p(q))\ .
\]
\end{altenumerate}
\end{proposition}

Note in particular that, as predicted by Conjecture~\ref{ConjA}, the left side is in all cases independent of $\square$.

\begin{proof} We give a sketch of the proof. The map $\varphi: \hat{R}\to \hat{R}$ makes $\hat{R}$ a free $\hat{R}$-module with basis given by $T_1^{a_1}\cdots T_d^{a_d}$ for $0\leq a_i\leq p-1$. This induces a similar direct sum decomposition of $q\!\op-\!\Omega^\bullet_{R[[q-1]]/\bb Z[[q-1]],\square}/\Phi_p(q)$; indeed, the direct sum decomposition of $\hat{R}$ is preserved by $\nabla_q$. Now one checks that all summands indexed by $a_i$ which are not all $0$ are acyclic. This can be done modulo $q-1$, where it reduces to a similar verification for $\Omega^\bullet_{(R/p)/\bb F_p}$. Thus, only the summand for $a_1=\ldots=a_d=0$ remains; but this summand has trivial differentials, which gives the required identifications.
\end{proof}

We note that this proposition provides a lift of the Cartier isomorphism to mixed characteristic: Recall that for any smooth $\bb F_p$-algebra $R_0$, there are canonical isomorphisms
\[
\Omega^i_{R_0/\bb F_p}\cong H^i(\Omega^\bullet_{R_0/\bb F_p})\ ,
\]
where the map is given by ``$\frac{\varphi}{p^i}$''. The proposition says that a similar result holds true over $\bb Z$, up to replacing the de~Rham complex by its $q$-deformation evaluated at $q=\zeta_p$. In particular, there is a tight relation between characteristic $p$ and $q$-deformations at $p$-th roots of unity, resembling a well-known phenomenon in the theory of quantum groups, cf.~\cite{Lusztig}.

\section{Known results}\label{sec:known}

In this section, we explain what is known. First, we note that the interesting things happen after $p$-adic completion. Indeed, after extending to $\bb Q[[q-1]]$, the $q$-de~Rham complex is quasi-isomorphic to the constant extension of the de~Rham complex:

\begin{lemma}\label{compqdRdR} Let $R$ be a smooth $\bb Z$-algebra with framing $\square$. Then there is a canonical isomorphism
\[
q\!\op-\!\Omega^\bullet_{R[[q-1]]/\bb Z[[q-1]]}\hat{\otimes}_{\bb Z[[q-1]]} \bb Q[[q-1]]\cong \Omega^\bullet_{R/\bb Z}\hat{\otimes}_{\bb Z} \bb Q[[q-1]]\ ,
\]
where both tensor products are $(q-1)$-adically completed.
\end{lemma}

In the case $R=\bb Z[T]$, this reduces to the observation that $[n]_q$ and $n$ differ (multiplicatively) by a unit in $\bb Q[[q-1]]$.

\begin{proof} Over $\bb Q[[q-1]]$, one can write down a Taylor series expressing the $q$-derivative $\nabla_q$ in terms of the usual derivative. Let $\nabla_{q,i}$ denote the $i$-th $q$-derivative, and $\nabla_i$ the $i$-th derivative. Then
\[
\nabla_{q,i} = \sum_{n\geq 1} \frac{\log(q)^n}{n!(q-1)}\underbrace{\nabla_i(T_i\nabla_i(\cdots (T_i\nabla_i)\cdots ))}_{n\ {\rm occurences\ of}\ \nabla_i}\ ,
\]
cf.~\cite[Lemma 12.3]{BMS}. Using this, one can build an isomorphism as in the lemma, cf.~\cite[Corollary 12.4]{BMS}.
\end{proof}

Now fix a prime number $p$. Then the $p$-adic completion of $q\!\op-\!\Omega_{R/\bb Z}$ should depend only on the $p$-adic completion $\hat{R}$ of $R$.

The results of \cite{BMS} are in a slightly different setting. Namely, we fix a complete algebraically closed extension $C$ of $\bb Q_p$, for example $C=\bb C_p$, and let $\roi = \roi_C$ be its ring of integers. Fontaine associated with $C$ the ring $A_\mathrm{inf}$ defined as $A_\mathrm{inf} = W(\roi^\flat)$, where $\roi^\flat = \varprojlim_\phi \roi/p$ is the ``tilt'' of $\roi$. Then $\roi^\flat$ is a perfect ring of characteristic $p$, and is the ring of integers in a complete algebraically closed field $C^\flat$ of characteristic $p$. In particular, the Frobenius $\phi$ of $\roi^\flat$ is an automorphism, and induces by functoriality an automorphism $\phi$ of $A_\mathrm{inf}$. There is a natural surjective map $\theta: A_\mathrm{inf}\to \roi$ whose kernel is generated by a non-zero-divisor $\xi\in A_\mathrm{inf}$.

In fact, if one chooses a system of primitive $p$-power roots of unity $1,\zeta_p,\zeta_{p^2},\ldots\in \roi$, these induce an element $\epsilon=(1,\zeta_p,\zeta_{p^2},\ldots)\in \varprojlim_\phi \roi/p= \roi^\flat$ with Teichm\"uller lift $[\epsilon]\in W(\roi^\flat)=A_\mathrm{inf}$, and one can choose $\xi = \frac{[\epsilon]-1}{[\epsilon^{1/p}]-1}$. Moreover, one gets a map $\bb Z_p[[q-1]]\to A_\mathrm{inf}$ sending $q$ to $[\epsilon]$, making $A_\mathrm{inf}$ a faithfully flat $\bb Z_p[[q-1]]$-algebra. Note that this map connects the formal variable $q$ from above with the roots of unity in algebraic extensions of $\bb Z_p$; this indicates that the $q$-deformation of this paper reflects some inner arithmetic of $\bb Z$.

Let us call a $p$-adically complete $\roi$-algebra $S$ smooth if it is the $p$-adic completion of some smooth $\roi$-algebra; equivalently, by a theorem of Elkik, \cite{Elkik}, if $S$ is a $p$-adically complete flat $\roi$-algebra such that $S/p$ is smooth over $\roi/p$. In this situation, one can define a variant of the $q$-de~Rham complex, as follows.

Fix a framing $\square: \roi\langle T_1,\ldots,T_d\rangle\to S$; here, $\roi\langle T_1,\ldots,T_d\rangle$ denotes the $p$-adic completion of $\roi[T_1,\ldots,T_d]$, and $\square$ is assumed to be \'etale (in the sense that it is flat and \'etale modulo $p$). In this situation, the analogue of $\bb Z[T_1,\ldots,T_d][[q-1]]$ is played by the $(p,\xi)$-adic completion $A_\inf\langle T_1,\ldots,T_d\rangle$ of $A_\inf[T_1,\ldots,T_d]$; effectively, we replace the surjection $\bb Z[[q-1]]\to \bb Z$ by $\theta: A_\inf\to \roi$. Now, for $i=1,\ldots,d$, we have the automorphism $\gamma_i$ of $A_\inf\langle T_1,\ldots,T_d\rangle$ sending $T_i$ to $[\epsilon]T_i$, and $T_j$ to $T_j$ for $j\neq i$. The \'etale map $\square: \roi\langle T_1,\ldots,T_d\rangle\to S$ deforms uniquely to an \'etale map $A_\inf\langle T_1,\ldots,T_d\rangle\to A(S)^\square$, and the automorphisms $\gamma_i$ lift uniquely to automorphisms of $A(S)^\square$. In this situation, we have the $q$-de~Rham complex, where $q=[\epsilon]$,
\[
q\!\op-\!\Omega^\bullet_{A(S)^\square/A_\inf} = A(S)^\square\buildrel\nabla_q\over\to \Omega^1_{A(S)^\square/A_\inf}\to\ldots\to \Omega^d_{A(S)^\square/A_\inf}\ ;
\]
here, all $\Omega^i$ are understood to be continuous K\"ahler differentials. As before,
\[
\nabla_q(f) = \sum_{i=1}^d \frac{\gamma_i(f)-f}{qT_i-T_i} dT_i
\]
denotes the $q$-derivative.

\begin{theorem}[\cite{BMS}]\label{thmA} There is a functor $S\mapsto A\Omega_S$ from the category of $p$-adically complete smooth $\roi$-algebras $S$ to the $\infty$-category of $E_\infty$-$A_\inf$-algebras, such that for any choice of framing $\square$ sending $T_i$ to an invertible function in $S$, $A\Omega_S$ is computed by $q\!\op-\!\Omega^\bullet_{A(S)^\square/A_\inf}$.
\end{theorem}

It should be possible to remove the assumption that $\square$ sends $T_i$ to an invertible function by using the $v$-topology from~\cite{ScholzeLectureNotes} in place of the pro-\'etale topology employed in \cite{BMS}.

This proves in particular that Conjecture~\ref{ConjA} holds true after (completed) base change to $A_\inf$. We expect the following compatibility:

\begin{conjecture}\label{ConjC} Let $R$ be a smooth $\bb Z$-algebra, and let $R_\roi$ be the $p$-adic completion of $R\otimes_{\bb Z} \roi$. Then, for the $q$-de~Rham complex $q\!\op-\!\Omega_R$ given by Conjecture~\ref{ConjA}, one has (functorially in $R$)
\[
q\!\op-\!\Omega_R\hat{\otimes}_{\bb Z[[q-1]]} A_\inf\cong A\Omega_{R_\roi}\ ,
\]
where the tensor product is $(p,\xi)$-adically completed, and $q\mapsto [\epsilon]$.\footnote{This map $\bb Z[[q-1]]\to A_\inf$ depends on a choice of $p$-power roots of unity, while the right-hand side does not. We expect that $q\!\op-\!\Omega_R$ admits semilinear operations, detailed in Section~\ref{sec:operations} below, that show directly that this base change is canonically independent of the choice.}
\end{conjecture}

We note that this is clear after the choice of coordinates; the conjecture says that the isomorphism commutes with coordinate transformations. Using the comparison between \'etale and singular cohomology, the next theorem states that Conjecture~\ref{ConjB} holds true in this context:

\begin{theorem}\label{thmB} Let $X$ be a proper smooth (formal) scheme over $\roi$ with (rigid-analytic) generic fibre $X_C$ over $C$. The hypercohomology groups
\[
H^i_{A_\inf}(X) := \bb H^i(X,A\Omega_X)
\]
are finitely presented $A_\inf$-modules, and after inverting $q-1=[\epsilon]-1$, there is a canonical isomorphism
\[
H^i_{A_\inf}(X)[\tfrac 1{[\epsilon]-1}]\cong H^i_\sub{\'et}(X_C,\bb Z_p)\otimes_{\bb Z_p} A_\inf[\tfrac 1{[\epsilon]-1}]\ .
\]
\end{theorem}

In particular, Conjecture~\ref{ConjB} follows from Conjecture~\ref{ConjC}. In our presentation, this theorem may be the most surprising, as it claims a direct relation between ($q$-)differentials and \'etale cohomology, which is different from the comparison isomorphism over the complex numbers coming from integration of differential forms along cycles. In the next section, we sketch the proofs of these theorems.

\section{The key ideas}\label{sec:ideas}

As in the previous section, let $\roi$ be the ring of integers in a complete and algebraically closed extension $C$ of $\bb Q_p$. Let $S$ be a $p$-adically complete smooth $\roi$-algebra as in the previous section. The construction of $A\Omega_S$ relies on a relative version of the construction of $A_\inf$, which was used extensively by Faltings, e.g.~in~\cite{FaltingsAlmostEtale}.

Assume that there is an \'etale map $\square: \roi[T_1,\ldots,T_d]\to S$ sending all $T_i$ to invertible elements. Moreover, we assume that $\Spec S$ is connected. Let $\bar{S}$ be the integral closure of $S$ in the maximal pro-finite \'etale extension of $S[\tfrac 1p]$. Then (the $p$-adic completion $\hat{\bar{S}}$ of) $\bar{S}$ is ``perfectoid'', meaning (essentially) that Frobenius induces an isomorphism between $\bar{S}/p^{1/p}$ and $\bar{S}/p$. Note that extracting all $p$-power roots of all $T_i$ defines a pro-finite \'etale covering of $S[\tfrac 1p]$; let $S_\infty\subset \bar{S}$ be the integral closure of $S$ in $S[\tfrac 1p,T_1^{1/p^\infty},\ldots,T_d^{1/p^\infty}]$. Then already (the $p$-adic completion $\hat{S}_\infty$ of) $S_\infty$ is perfectoid. Note that $\bar{S}$ is canonical, while $S_\infty$ depends on the choice of coordinates. Also, both extensions $\bar{S}[\tfrac 1p]/S[\tfrac 1p]$ and $S_\infty[\tfrac 1p]/S[\tfrac 1p]$ are Galois; let $\Delta\rightarrow \Gamma\cong \bb Z_p^d$ be their respective Galois groups.

The following deep theorem, known as Faltings's almost purity theorem, is critical to the proof.

\begin{theorem}[{\cite[Theorem 3.1]{FaltingsJAMS}}]\label{thmalmostpurity} The extension $\bar{S}/S_\infty$ is ``almost finite \'etale''. In particular, the map of Galois cohomology groups
\[
H^i(\Gamma,S_\infty)\to H^i(\Delta,\bar{S})
\]
is an almost isomorphism, i.e.~the kernel and cokernel are killed by $p^{1/n}$ for all $n\geq 1$.
\end{theorem}

The theorem is a weak version of Abhyankar's lemma in this setup: Extracting $p$-power roots of all $T_i$ kills almost all the ramification along the special fiber of any finite \'etale cover of $S[\tfrac 1p]$.

We apply this theorem not to $\bar{S}$ and $S_\infty$ directly, but rather to versions of Fontaine's $A_\inf$ ring constructed from them. Namely, let
\[\begin{aligned}
A_\inf(\bar{S}) &= W(\bar{S}^\flat)\ ,\ \bar{S}^\flat = \varprojlim_\phi \bar{S}/p\ ,\\
A_\inf(S_\infty) &= W(S_\infty^\flat)\ ,\ S_\infty^\flat = \varprojlim_\phi S_\infty/p\ .
\end{aligned}\]
As for $A_\inf = A_\inf(\roi)$, there is a surjective map $\theta: A_\inf(\bar{S})\to \hat{\bar{S}}$ (resp. $\theta: A_\inf(S_\infty)\to \hat{S}_\infty$) whose kernel is generated by the non-zero-divisor $\xi$.

We note that there are some special elements in $S_\infty^\flat\subset \bar{S}^\flat$. Namely, we have the elements $\epsilon=(1,\zeta_p,\zeta_{p^2},\ldots)$, $T_i^\flat = (T_i,T_i^{1/p},T_i^{1/p^2},\ldots)$ consisting of sequences of $p$-th roots. Their Teichm\"uller lifts define elements $[\epsilon],[T_i^\flat]\in A_\inf(S_\infty)$.

The following lemma is critical. It relates the deformation $A(S)^\square$ of $S$ over $A_\inf$ to $A_\inf(S_\infty)$. Recall that $A(S)^\square$ is (formally) \'etale over $A_\inf\langle T_1,\ldots,T_d\rangle$. In order to better distinguish between different objects, let us rewrite the variables $T_i$ in this algebra as $U_i$, so $A(S)^\square$ is formally \'etale over $A_\inf\langle U_1,\ldots,U_d\rangle$, lifting the formally \'etale extension $S$ of $\roi\langle T_1,\ldots,T_d\rangle$ via the map $A_\inf\langle U_1,\ldots,U_d\rangle\to \roi\langle T_1,\ldots,T_d\rangle$ sending $U_i$ to $T_i$.

\begin{lemma} The map
\[
A_\inf\langle U_1^{1/p^\infty},\ldots,U_d^{1/p^\infty}\rangle\to A_\inf(S_\infty)\ ,\ U_i^{1/p^n}\mapsto [(T_i^\flat)^{1/p^n}]\ ,
\]
extends to a unique isomorphism
\[
A_\inf\langle U_1^{1/p^\infty},\ldots,U_d^{1/p^\infty}\rangle\hat{\otimes}_{A_\inf\langle U_1,\ldots,U_d\rangle} A(S)^\square\cong A_\inf(S_\infty)
\]
commuting with the natural maps to $\hat{S}_\infty$ explained below.
\end{lemma}

Here, on the left side, the natural map
\[
A_\inf\langle U_1^{1/p^\infty},\ldots,U_d^{1/p^\infty}\rangle\hat{\otimes}_{A_\inf\langle U_1,\ldots,U_d\rangle} A(S)^\square\to \hat{S}_\infty
\]
sends $U_i^{1/p^n}$ to $T_i^{1/p^n}$ (in particular, $U_i$ to $T_i$), and $A(S)^\square\to S$. On the right side, it is given by $\theta$.

In particular, we see that the functions $U_i$, previously denoted by $T_i$, in the deformation $A(S)^\square$ are related to systems of $p$-power roots of the functions $T_i\in S$. This observation seems to be critical to an understanding of the phenomena in this paper.

Now we can construct $A\Omega_S$, up to some small $(q-1)$-torsion, where $q=[\epsilon]$.

\begin{theorem}[{\cite{BMS}}]\label{thmweak} Let $S$ be a smooth $\roi$-algebra with framing $\square$, sending $T_i$ to invertible elements of $S$, as above.
\begin{altenumerate}
\item[{\rm (i)}] The kernel and cokernel of the map
\[
H^i(\Gamma,A_\inf(S_\infty))\to H^i(\Delta,A_\inf(\bar{S}))
\]
are killed by $q-1$.
\item[{\rm (ii)}] The map
\[
H^i(\Gamma,A(S)^\square)\to H^i(\Gamma,A_\inf(S_\infty))
\]
is injective, with cokernel killed by $q-1$.
\item[{\rm (iii)}] There is a natural map
\[
q\!\op-\!\Omega^\bullet_{A(S)^\square/A_\inf}\to R\Gamma(\Gamma,A(S)^\square)
\]
such that the induced map cohomology has kernel and cokernel killed by $(q-1)^d$.
\end{altenumerate}
\end{theorem}

\begin{proof} Let us give a brief sketch. Part (i) is a consequence of Faltings's almost purity theorem. For part (ii), note that the isomorphism
\[
A_\inf\langle U_1^{1/p^\infty},\ldots,U_d^{1/p^\infty}\rangle\hat{\otimes}_{A_\inf\langle U_1,\ldots,U_d\rangle} A(S)^\square\cong A_\inf(S_\infty)
\]
makes $A_\inf(S_\infty)$ topologically free over $A(S)^\square$, with basis given by monomials $U_1^{a_1}\cdots U_d^{a_d}$, with $a_i\in [0,1)\cap \bb Z[\tfrac 1p]$. This decomposition is $\Gamma$-equivariant. In particular, $A(S)^\square$ itself splits off as a $\Gamma$-equivariant direct summand, proving (split) injectivity. The cokernel is computed by the cohomology of the terms indexed by non-integral exponents. This cohomology can be shown to be killed by $q-1$ by a direct computation, cf.~\cite[Lemma 9.6]{BMS}.

Finally, for part (iii), note that cohomology with coefficients in $\Gamma = \bb Z_p^d$ can be computed by a Koszul complex. This is given by
\[
A(S)^\square\to (A(S)^\square)^d\to (A(S)^\square)^{d\choose 2}\to \ldots \to (A(S)^\square)^{d\choose d}\ ,
\]
where the components of the first differential are given by $\gamma_i-1: A(S)^\square\to A(S)^\square$, $i=1,\ldots,d$. Similarly, we can write
\[
q\!\op-\!\Omega^\bullet_{A(S)^\square/A_\inf} = A(S)^\square\to (A(S)^\square)^d\to (A(S)^\square)^{d\choose 2}\to \ldots \to (A(S)^\square)^{d\choose d}\ ,
\]
where we use the logarithmic differentials $d\log(U_i) = \frac{dU_i}{U_i}$ as a basis for $\Omega^1_{A(S)^\square/A_\inf}$. In this normalization, the components of the first differential are given by $\frac{\gamma_i-1}{q-1}$, $i=1,\ldots,d$. We see that multiplication by $(q-1)^i$ in degree $i$ defines a map of complexes
\[
q\!\op-\!\Omega^\bullet_{A(S)^\square/A_\inf}\to R\Gamma(\Gamma,A(S)^\square)\ ,
\]
whose cone is killed by $(q-1)^d$; in particular, we get the desired result on the level of cohomology groups.
\end{proof}

Similar computations were already done by Faltings, although they were not expressed in the language of the $q$-de~Rham complex. The new observation in \cite{BMS} was that it was possible to remove all $(q-1)$-torsion in the previous result; the new idea was to use the following operation on the derived category (first defined by Berthelot--Ogus, \cite[Chapter 8]{BerthelotOgus}).

\begin{proposition} Let $A$ be a ring with a non-zero-divisor $f\in A$. Let $C^\bullet$ be a complex of $f$-torsion free $A$-modules. Define a new complex of ($f$-torsion free) $A$-modules $\eta_f C^\bullet$ as a subcomplex of $C^\bullet[\tfrac 1f]$ by
\[
(\eta_f C)^i = \{x\in f^i C^i\mid dx\in f^{i+1} C^{i+1}\}\ .
\]
If $C_1^\bullet\to C_2^\bullet$ is a quasi-isomorphism of $f$-torsion free $A$-modules, then $\eta_f C_1^\bullet\to \eta_f C_2^\bullet$ is also a quasi-isomorphism; thus, this operation passes to the derived category $D(A)$ of $A$-modules, and defines a functor $L\eta_f: D(A)\to D(A)$.
\end{proposition}

Now we can give the definition of $A\Omega_S$.

\begin{definition} The complex $A\Omega_S$ is given by
\[
A\Omega_S = L\eta_{q-1} R\Gamma(\Delta,A_\inf(\bar{S}))\ .
\]
\end{definition}

This gives the following refined version of Theorem~\ref{thmweak}.

\begin{theorem}[{\cite{BMS}}]\label{thmstrong} Let $S$ be a smooth $\roi$-algebra with framing $\square$, sending $T_i$ to invertible elements of $S$, as above.
\begin{altenumerate}
\item[{\rm (i)}] The map
\[
L\eta_{q-1} R\Gamma(\Gamma,A_\inf(S_\infty))\to L\eta_{q-1} R\Gamma(\Delta,A_\inf(\bar{S})) = A\Omega_S
\]
is a quasi-isomorphism.
\item[{\rm (ii)}] The map
\[
L\eta_{q-1} R\Gamma(\Gamma,A(S)^\square)\to L\eta_{q-1} R\Gamma(\Gamma,A_\inf(S_\infty))
\]
is a quasi-isomorphism.
\item[{\rm (iii)}] There is a natural quasi-isomorphism
\[
q\!\op-\!\Omega^\bullet_{A(S)^\square/A_\inf}\to L\eta_{q-1} R\Gamma(\Gamma,A(S)^\square)\ .
\]
\end{altenumerate}
\end{theorem}

Combining parts (i), (ii) and (iii), we get the desired quasi-isomorphism
\[
q\!\op-\!\Omega^\bullet_{A(S)^\square/A_\inf}\to A\Omega_S\ ,
\]
proving that $q\!\op-\!\Omega^\bullet_{A(S)^\square/A_\inf}$ is canonical, as desired.

\begin{proof} The proof builds on Theorem~\ref{thmweak}, and refines the arguments in each step. The hardest part is (i). The surprising part of the proof is that one can prove part (i) without establishing any further properties of the completely inexplicit $R\Gamma(\Delta,A_\infty(\bar{S}))$ besides what follows from Faltings's almost purity theorem. The phenomenon at work here is that in some circumstances, the operation $L\eta_{q-1}$ turns almost quasi-isomorphisms into quasi-isomorphisms. This is not true in general, but it is enough to know that the source is nice.
\end{proof}

Note that the definition of $A\Omega_S$ is by some form of (finite) \'etale cohomology of $S[\tfrac 1p]$; it is through a computation that this gets related to differential forms. However, one takes cohomology with coefficients in the big ring $A_\inf(\bar{S})$, so it is not a priori clear that one gets the expected relation to \'etale cohomology as in Theorem~\ref{thmB}. Theorem~\ref{thmB} follows from a comparison result of Faltings, \cite[\S 3, Theorem 8]{FaltingsAlmostEtale}, \cite[Theorem 1.3]{scholzepadichodge}, saying roughly that on a proper smooth space, the \'etale cohomology with coefficients in the relative $A_\inf$-construction agrees (almost) with the \'etale cohomology with coefficients in the constant ring $A_\inf = A_\inf(\roi)$.

\section{Operations on $q$-de~Rham cohomology}\label{sec:operations}

We expect that there are many semilinear operations on the $q$-de~Rham complex, and thus on $q$-de~Rham cohomology. This is related to some investigations by Kedlaya on ``doing $p$-adic Hodge theory for all $p$ at once'', \cite{kedlayaglobalhodge}.

To motivate the extra structure we expect on $q$-de~Rham cohomology, we first relate it to some known constructions in abstract $p$-adic Hodge theory. Assume that $X$ is smooth and proper over $\bb Z[\tfrac 1N]$, and $p$ is a prime that does not divide $N$. Assume for simplicity that $H^i_{q\!\op-\!\dR}(X)\otimes_{\bb Z[\frac 1N][[q-1]]} \bb Z_p[[q-1]]$ is a finite free $\bb Z_p[[q-1]]$-module. In particular, by Conjecture~\ref{ConjB} (cf.~Theorem~\ref{thmB}), it follows that $H^i_\sub{\'et}(X_{\bar{\bb Q}_p},\bb Z_p)$ is a finite free $\bb Z_p$-module. By the crystalline comparison theorem, $H^i_\sub{\'et}(X_{\bar{\bb Q}_p},\bb Z_p)$ is a lattice in a crystalline representation of $G_{\bb Q_p} = \mathrm{Gal}(\bar{\bb Q}_p/\bb Q_p)$.

The theory of Wach modules attaches to $H^i_\sub{\'et}(X_{\bar{\bb Q}_p},\bb Z_p)$ a finite free $\bb Z_p[[q-1]]$-module $M_p$ equipped with many semilinear operators, to be detailed below. We expect that
\[
M_p=H^i_{q\!\op-\!\dR}(X)\otimes_{\bb Z[\frac 1N][[q-1]]} \bb Z_p[[q-1]]\ ,
\]
and that these semilinear operators are intrinsically defined on $H^i_{q\!\op-\!\dR}(X)\otimes_{\bb Z[\tfrac 1N][[q-1]]} \bb Z_p[[q-1]]$.

There are the following semilinear operators on $M_p$. First, there is a Frobenius operator $\phi = \phi_p$. This operator is semilinear with respect to the map $\bb Z_p[[q-1]]\to \bb Z_p[[q-1]]$ which sends $q$ to $q^p$. This should correspond to the following operation on the $q$-de~Rham complex.

\begin{conjecture}\label{conjphi} Let $R$ be a smooth $\bb Z$-algebra, with $p$-adic completion $\hat{R}$. On the $p$-adic completion $\widehat{q\!\op-\!\Omega_R}$ of the $q$-de~Rham complex $q\!\op-\!\Omega_R$, there is a $\bb Z[[q-1]]$-semilinear (w.r.t.~$q\mapsto q^p$) endomorphism (of $E_\infty$-algebras)
\[
\phi_p: \widehat{q\!\op-\!\Omega_R}\to \widehat{q\!\op-\!\Omega_R}\ .
\]
If $\square: \bb Z[T_1,\ldots,T_d]\to R$ is a framing, then under the quasi-isomorphism $q\!\op-\!\Omega_R\cong q\!\op-\!\Omega^\bullet_{R[[q-1]]/\bb Z[[q-1]]}$, $\phi_p$ is induced by the map
\[
\phi\otimes (q\mapsto q^p): \hat{R}[[q-1]]\to \hat{R}[[q-1]]
\]
sending $T_i$ to $T_i^p$, and $q$ to $q^p$; more precisely, on $\hat{R}$, it is given by the Frobenius lift $\phi$ from (the discussion before) Proposition~\ref{propcartier}.
\end{conjecture}

Note that the map $\phi\otimes (q\mapsto q^p): \hat{R}[[q-1]]\to \hat{R}[[q-1]]$ commutes with the natural automorphisms $\gamma_i$: On the source, they send $T_i$ to $qT_i$, and on the target, they send $T_i^p$ to $q^p T_i^p$. Thus, the map $\hat{R}[[q-1]]\to \hat{R}[[q-1]]$ does indeed induce an endomorphism of the $p$-completion of $q\!\op-\!\Omega^\bullet_{R[[q-1]]/\bb Z[[q-1]]}$ (involving multiplication by $(\frac{q^p-1}{q-1})^i$ in degree $i$).

On the other hand, there is a semilinear $\bb Z_p^\times$-action on the Wach module $M_p$. For any $a\in \bb Z_p^\times$, this is semilinear with respect to the map $\bb Z_p[[q-1]]\to \bb Z_p[[q-1]]$ which sends $q$ to
\[
q^a = (1+(q-1))^a = \sum_{n\geq 0} {a\choose n} (q-1)^n\ ;
\]
here, ${a\choose n}\in \bb Z_p$ is a well-defined element. We expect that this operation also comes from an operation on the $p$-adic completion of the $q$-de~Rham complex. Even better, if $a\in \bb Z_p^\times\cap \bb N$ is an integer, we expect that this operation is defined before taking the $p$-adic completion.

\begin{conjecture}\label{conjgamma} Let $R$ be a smooth $\bb Z$-algebra, and let $a\neq 0$ be an integer. Assume that $a$ is invertible on $R$. Then there is a $\bb Z[[q-1]]$-semilinear (w.r.t.~$q\mapsto q^a$) automorphism (of $E_\infty$-algebras)
\[
\gamma_a: q\!\op-\!\Omega_R\cong q\!\op-\!\Omega_R\ .
\]
If $\square: \bb Z[T_1,\ldots,T_d]\to R$ is a framing, then under the quasi-isomorphism $q\!\op-\!\Omega_R\cong q\!\op-\!\Omega^\bullet_{R[[q-1]]/\bb Z[[q-1]]}$, $\gamma_a$ is given by a map of complexes
\[\xymatrix{
R[[q-1]] \ar[r]^{\nabla_q}\ar[d]_{q\mapsto q^a} & \Omega^1_{R[[q-1]]/\bb Z[[q-1]]}\ar[r]^{\nabla_q}\ar[d] & \ldots\\
R[[q-1]] \ar[r]^{\nabla_q} & \Omega^1_{R[[q-1]]/\bb Z[[q-1]]}\ar[r] & \ldots\ ,
}\]
where the left vertical map is the identity on $R$ and sends $q$ to $q^a$.
\end{conjecture}

In principle, one can make the map of complexes explicit. For this, one observes that one can express $\nabla_{q^a}$ through $\nabla_q$ (and conversely) as long as one works over $\bb Z[\tfrac 1a][[q-1]]$. Related to this, observe that the endomorphism $\bb Z[\tfrac 1a][[q-1]]\to \bb Z[\tfrac 1a][[q-1]]$ sending $q$ to $q^a$ is an automorphism.

Also note that if $a$ is not invertible in $R$, one can still apply $\gamma_a$ to
\[
q\!\op-\!\Omega_{R[\frac 1a]} = q\!\op-\!\Omega_R\hat{\otimes}_{\bb Z[[q-1]]} \bb Z[\tfrac 1a][[q-1]]\ .
\]
In particular, for any prime $p$, we have two associated automorphisms $\phi_p$, $\gamma_p$: The first acts on the $p$-adic completion, while the second acts after inverting $p$. One can compare the actions after base change to $\bb Q_p[[q-1]]$, i.e.~on the $(q-1)$-adic completion of $\widehat{q\!\op-\!\Omega_R}[\tfrac 1p]$. Over this ring, the $q$-de~Rham complex becomes isomorphic to the constant extension of de~Rham cohomology by Lemma~\ref{compqdRdR}, which in this case agrees with the crystalline cohomology of $R/p$. On crystalline cohomology of $R/p$, one also has a Frobenius action (which is Frobenius-semilinear on $\hat{R}$); this can be extended $\bb Z[[q-1]]$-linearly to an action on the $(q-1)$-adic completion of $\widehat{q\!\op-\!\Omega_R}[\tfrac 1p]$. We expect that the actions of $\phi_p$ and $\gamma_p$ differ by the action of this crystalline Frobenius.

We note in particular that after extending the scalars to $\bb Q[[q-1]]$, the $q$-de~Rham complex acquires an action by $\bb Q^\times$. Under the identification
\[
q\!\op-\!\Omega_R\hat{\otimes}_{\bb Z[[q-1]]} \bb Q[[q-1]]\cong \Omega^\bullet_{R/\bb Z}\hat{\otimes}_{\bb Z} \bb Q[[q-1]]\ ,
\]
this corresponds to the action on the second factor, via the semilinear automorphisms $\bb Q[[q-1]]\to \bb Q[[q-1]]$, $q\mapsto q^a$. Note that one also has an identification $\bb Q[[q-1]]\cong \bb Q[[h]]$, where $q=\exp(h)$; under this identification, the $\bb Q^\times$-action becomes the scaling action $h\mapsto ah$, $a\in \bb Q^\times$.

\begin{example}\label{ex:tatetwist} As an example, let us discuss the Tate twist in this setup. Namely, let $\bb Z[[q-1]]\{-1\} := H^2_{q\!\op-\!\dR}(\bb P^1_{\bb Z})$.\footnote{The notation $\{1\}$ in place of $(1)$ for the Tate twist follows notation introduced in \cite{BMS}. The reason is that $A_\inf\{1\}:=A_\inf\otimes_{\bb Z[[q-1]]} \bb Z[[q-1]]\{1\}$ is not equal to $A_\inf(1) := A_\inf\otimes_{\bb Z_p} \bb Z_p(1)$, but there is a canonical map $A_\inf(1)\to A_\inf\{1\}$ with image given by $(q-1)A_\inf\{1\}$; there is thus some need to distinguish between those two similar but different objects.} By comparison with de~Rham cohomology, it follows that $\bb Z[[q-1]]\{-1\}$ is a free $\bb Z[[q-1]]$-module of rank $1$. We expect, cf.~Section~\ref{sec:examples} below, that there is a canonical basis element $e\in \bb Z[[q-1]]\{-1\}$ on which the actions of $\phi_p$ and $\gamma_a$ are given as follows:
\[\begin{aligned}
\phi_p(e) &= \frac{q^p-1}{q-1} e = [p]_q e\ ,\\
\gamma_a(e) &= \frac 1a \frac{q^a-1}{q-1} e = \frac{[a]_q}{a} e\ .
\end{aligned}\]
Recall that after base extension to $\bb Q[[q-1]]$, $q$-de~Rham cohomology becomes isomorphic to the constant extension of de~Rham cohomology. In this case, the resulting isomorphism
\[
\bb Z[[q-1]]\{-1\}\otimes_{\bb Z[[q-1]]} \bb Q[[q-1]]\cong H^2_\dR(\bb P^1_{\bb Z})\otimes_{\bb Z} \bb Q[[q-1]] = \bb Q[[q-1]] \omega\ ,
\]
where $\omega\in H^2_\dR(\bb P^1_{\bb Z})$ is the standard generator, is given by sending $e$ to $\frac{q-1}{\log(q)} \omega$. Note that indeed
\[
\gamma_a\left(\frac{q-1}{\log(q)} \omega\right) = \frac{q^a-1}{\log(q^a)} \omega = \frac 1a \frac{q^a-1}{q-1}\cdot \frac{q-1}{\log(q)} \omega\ ,
\]
compatibly with $\gamma_a(e) = \frac 1a \frac{q^a-1}{q-1} e$.
\end{example}

\begin{remark} One may wonder whether the theory also connects to complex Hodge theory. As after base extension to $\bb C[[q-1]]\cong \bb C[[h]]$ (in fact, already to $\bb Q[[q-1]]$), one gets a constant extension of de~Rham cohomology, there are certainly ways to glue the $q$-de~Rham complex with structures from complex Hodge theory. One possibility is to interpret a Hodge structure as a $\bb G_m$-equivariant sheaf on the twistor $\bb P^1$, cf.~\cite{simpson}. By passing to the completion at $0$ (or $\infty$), one gets a $\bb C[[h]]$-module which is $\bb G_m$-equivariant under the scaling action, and in particular $\bb Q^\times$-equivariant.
\end{remark}

\section{$q$-connections}\label{sec:qconn}

We want to discuss briefly the expected relative theory. Roughly, we conjecture that the Gau\ss-Manin connection deforms to a $q$-connection.

More precisely, assume that $R$ is a smooth $\bb Z$-algebra, with framing $\square$, and $f: X\to \Spec R$ is a proper smooth morphism. For simplicity, we assume that all relative de Rham cohomology groups $H^i_\dR(X/R) = \bb H^i(X,\Omega^\bullet_{X/R})$ are locally free on $\Spec R$; this can be ensured after inverting some big enough integer.

\begin{conjecture}\label{conjsmooth} Fix a framing $\square: \bb Z[T_1,\ldots,T_d]\to R$ of $R$. There is a natural functor $S\mapsto q\!\op-\!\Omega_{S/R,\square}$ from the category of smooth $R$-algebras $S$ to the category of $E_\infty$-$R[[q-1]]$-algebras, with $q\!\op-\!\Omega_{S/R,\square}/(q-1) = \Omega_{S/R}$.
\end{conjecture}

In coordinates, one can write down $q\!\op-\!\Omega_{S/R,\square}$ as before; we leave the details to the reader. Note that $q\!\op-\!\Omega_{S/R,\square}$ should be independent of a choice of coordinates for $S$ (over $R$); it does however depend on the choice of coordinates for the base $R$. The choice of $\square$ gives a ``stupid truncation map''
\[
q\!\op-\!\Omega_R\cong q\!\op-\!\Omega^\bullet_{R[[q-1]]/\bb Z[[q-1]]}\to R[[q-1]]\ ,
\]
which is a map of $E_\infty$-$\bb Z[[q-1]]$-algebras. There should be a natural map
\[
q\!\op-\!\Omega_S\hat{\otimes}_{q\!\op-\!\Omega_R} R[[q-1]]\to q\!\op-\!\Omega_{S/R,\square}\ ,
\]
where the tensor product is taken in the category of $E_\infty$-algebras, and $(q-1)$-adically completed. We expect that this map is an equivalence after $p$-adic completion for all $p$ (as the similar map
\[
\Omega^\bullet_S\otimes_{\Omega^\bullet_R} R\to \Omega^\bullet_{S/R}
\]
is an equivalence after $p$-adic completion for all $p$, but not in general over $\bb Q$). Again, after base change to $\bb Q[[q-1]]$, $q\!\op-\!\Omega_{S/R}$ should become isomorphic to $(\Omega_{S/R}^\bullet\otimes_{\bb Z} \bb Q)[[q-1]]$, and it should be possible to glue this together with the profinite completion of $q\!\op-\!\Omega_R\hat{\otimes}_{q\!\op-\!\Omega_R} R[[q-1]]$ to get $q\!\op-\!\Omega_{S/R,\square}$.

Assuming the conjecture, we can now define
\[
H^i_{q\!\op-\!\dR,\square}(X/R) = \bb H^i(X,q\!\op-\!\Omega_{X/R,\square})\ .
\]
In analogy with a result from \cite[Corollary 4.20, Theorem 14.5 (iii)]{BMS}, we expect the following conjecture.

\begin{conjecture} Assume that $H^i_\dR(X/R)$ is a finite projective $R$-module for all $i\in \bb Z$. Then $H^i_{q\!\op-\!\dR,\square}(X/R)$ is a finite projective $R[[q-1]]$-module for all $i\in \bb Z$.
\end{conjecture}

Moreover, $H^i_{q\!\op-\!\dR,\square}(X/R)$ should admit a $q$-connection in the following sense.

\begin{definition} Let $R$ be a smooth $\bb Z$-algebra with framing $\square:\bb Z[T_1,\ldots,T_d]\to R$. A finite projective $R[[q-1]]$-module with (flat) $q$-connection w.r.t.~the framing $\square$ is a finite projective $R[[q-1]]$-module $M$ equipped with $d$ commuting maps $\nabla_{q,1},\ldots,\nabla_{q,d}: M\to M$, satisfying the $q$-Leibniz rule for $f\in R$, $m\in M$:
\[
\nabla_{q,i}(fm) = \gamma_i(f) \nabla_{q,i}(m) + \nabla_{q,i}(f) m\ ,
\]
where as before $\nabla_{q,i}(f) = \frac{\gamma_i(f)-f}{qT_i-T_i}$.
\end{definition}

\begin{remark} This is an instance of a general definition that has been widely studied before; we refer to an article by Andr\'e, \cite{andreq}. In particular, it is proved there that the category of $q$-connections is naturally a symmetric monoidal category. It is however a nontrivial exercise to write down the symmetric monoidal tensor product for this category, as the naive formulas will not respect the $q$-Leibniz rule (due to its noncommutativity). This phenomenon is related to the fact that the $q$-de~Rham complex is an $E_\infty$-algebra, but it cannot be represented by a commutative differential graded algebra, cf.~\cite[Remark 7.8]{BMS}.

We warn the reader that the category of modules with $q$-connection is not directly related to the category of modules over the $E_\infty$-algebra $q\!\op-\!\Omega_R$. There is a functor (given by a $q$-de~Rham complex) from the category of modules with $q$-connection to the category of modules over $q\!\op-\!\Omega_R$, but it is not fully faithful. In fact, the same phenomenon happens classical for modules with integrable connection and modules over $\Omega_R$.
\end{remark}

In fact, we conjecture that the category of modules with $q$-connections is independent of the choice of coordinates:

\begin{conjecture}\label{conjmodule} Let $R$ be a smooth $\bb Z$-algebra. Then the symmetric monoidal category of finite projective $R[[q-1]]$-modules with $q$-connection w.r.t.~the framing $\square$ is canonically independent of the choice of $\square$.
\end{conjecture}

We warn the reader that we expect this equivalence to change the underlying $R[[q-1]]$-module of the $q$-connection. The situation is similar to the situation of crystals on smooth varieties in characteristic $p$: Once one chooses a deformation of the variety to $\bb Z_p$ (as is automatic by a choice of \'etale map to affine space), the category of crystals can be described as the category of modules with integrable connection on the deformation. By the theory of the crystalline site, this category is independent of the choice of coordinates. What we are missing in the $q$-de~Rham context is an analogue of the crystalline site.

\begin{remark} As in Section~\ref{sec:operations}, there should be further compatible $\phi_p$ and $\gamma_a$-actions on all objects considered in this section. It would be interesting to systematically study the resulting categories of $R[[q-1]]$-modules with $q$-connections and such actions as a category in its own right. For example, by the theory of relative $(\varphi,\Gamma)$-modules, it follows that for any $p$, it should admit a canonical functor to the category of $\bb Z_p$-local systems on $\Spec \hat{R}[\frac 1p]$, where $\hat{R}$ is the $p$-adic completion of $R$.
\end{remark}

\section{Examples}\label{sec:examples}

This section is unreasonably short, as the author is unable to compute any interesting example. It is not known to the author how to identify the $q$-de~Rham complexes $q\!\op-\!\Omega^\bullet_{R[[q-1]]/\bb Z[[q-1]],\square_1}$ and $q\!\op-\!\Omega^\bullet_{R[[q-1]]/\bb Z[[q-1]],\square_2}$ in the simplest case $R=\bb Z[T]$, with the two different framings $\square_1: \bb Z[T_1]\to \bb Z[T]$ sending $T_1$ to $T$, and $\square_2: \bb Z[T_2]\to \bb Z[T]$ sending $T_2$ to $T+1$. Note that the results of \cite{BMS} prove this independence after the completed base change along $\bb Z[[q-1]]\to A_\inf$, but it seems hard to make the proof explicit.

Roughly, in this example, in order to make the proof explicit, one would have to look at the integral closure of the algebra $\roi[T^{1/p^\infty},(1+T)^{1/p^\infty}]$ in $\bb C_p[T^{1/p^\infty},(1+T)^{1/p^\infty}]$, where $\roi$ is the ring of integers of $\bb C_p$.\footnote{To be in the setup of \cite{BMS}, one should also invert $T$ and $1+T$.} The almost purity theorem, Theorem~\ref{thmalmostpurity}, guarantees that this integral closure is significantly bigger than $\roi[T^{1/p^\infty},(1+T)^{1/p^\infty}]$; for example, it contains the element
\[
\frac{(1+T)^{1/p}-T^{1/p}-1}{p^{1/p}}\ ,
\]
as its $p$-th power can be checked to lie in $\roi[T^{1/p^\infty},(1+T)^{1/p^\infty}]$. It should be possible to write down the whole integral closure, and then compute the Galois cohomology of $\bb Z_p^2$ acting on it (via acting on the $p$-power roots $T^{1/p^n}$ and $(1+T)^{1/p^n}$). Applying $L\eta_{q-1}$ to the result would give the complex that admits natural quasi-isomorphisms from both $q$-de~Rham complexes.

One can do a computation for $\bb P^1_{\bb Z}$, or any smooth toric variety. For this, one notes that in the definition of the $q$-de~Rham complex, the only ingredient that depended on the choice of coordinates was the automorphism ``multiplication by $q$ in $i$-th coordinate''. Given a global action by a torus, these automorphisms can be defined globally, and one can expect that this gives the correct global $q$-de~Rham complex.

For $\bb P^1_{\bb Z}$, we get a global $q$-de~Rham complex
\[
q\!\op-\!\Omega^\bullet_{\bb P^1_{\bb Z[[q-1]]}/\bb Z[[q-1]]} = \roi_{\bb P^1_{\bb Z}}[[q-1]]\buildrel{\nabla_q}\over\to \Omega^1_{\bb P^1_{\bb Z}/\bb Z}[[q-1]]\ ,
\]
where $\nabla_q(f) = \frac{\gamma(f)-f}{q-1} d\log(\frac{x_0}{x_1})$, where $(x_0:x_1)$ are the coordinates on $\bb P^1$, and $\gamma$ is the automorphism of $\bb P^1_{\bb Z[[q-1]]}$ taking $(x_0:x_1)$ to $(qx_0:x_1)$. It follows that
\[
H^2_{q\!\op-\!\dR}(\bb P^1_{\bb Z}/\bb Z) = H^1(\bb P^1_{\bb Z},\Omega^1_{\bb P^1_{\bb Z}/\bb Z}[[q-1]])\cong \bb Z[[q-1]]\ .
\]
Via the usual Cech cover of $\bb P^1$, a generator of $H^1$ comes from the section $\omega=d\log(\frac{x_0}{x_1})$ of $\Omega^1_{\bb G_m/\bb Z}[[q-1]]$. One can compute the action of $\phi_p$, $\gamma_a$ on this class, resulting in the formulas given in Example~\ref{ex:tatetwist} above.

It would be interesting to compute $q$-de~Rham cohomology for elliptic curves. This could either be done for a single elliptic curve over an \'etale $\bb Z$-algebra, or for a family of elliptic curves, such as the Legendre family
\[
y^2 = x(x-1)(x-\lambda)\ ,\ \lambda\neq 0,1
\]
over $\bb P^1_{\bb Z}\setminus \{0,1,\infty\}$ (with its canonical framing). Note that in the latter case, the Gau\ss--Manin connection on de~Rham cohomology can be written down in a preferred basis, giving rise to the Picard--Fuchs equation. If there is a preferred element in the $q$-de~Rham cohomology of the universal elliptic curve, one could write down the $q$-difference equation that it satisfies, leading to a $q$-deformation of the Picard--Fuchs equation. Recall that the Picard--Fuchs equation is a particular kind of hypergeometric equation, and those admit natural $q$-deformations known as $q$-hypergeometric equations, studied for example in~\cite{andreq}. One may wonder whether there is a relation between these objects.

We warn the reader that if one knows the answer for the universal elliptic curve, it does not immediately give the answer for any individual elliptic curve: Although the category of modules with $q$-connection should be functorial in the base ring $R$, only certain functorialities are easy to write down in coordinates. For example, if $R=\bb Z[T^{\pm 1}]$ with its standard frame, then only the base change along $T=1$ should be equal to the naive base change $M\mapsto M\otimes_{\bb Z[T^{\pm 1}][[q-1]]} \bb Z[[q-1]]$. Note that for the ring $R=\bb Z[\lambda^{\pm 1},(1-\lambda)^{-1}]$ relevant for the Legendre family, even this base change is excluded.

The reader will have observed that whenever we could actually compute anything, it was not necessary to invoke power series in $q-1$.\footnote{In order to write down the $q$-de~Rham complex in coordinates for a general \'etale framing map, it was necessary to use power series in $q-1$, but this is a priori not saying anything about the underlying canonical object.} We do not know whether this is a reflection of not having computed sufficiently interesting objects, or of a more precise theory of $q$-deformations which works over a smaller ring than $\bb Z[[q-1]]$.

\section{Variants}\label{sec:variants}

In this paper, we worked under quite strong assumptions, in that we always assumed the base ring to be $\bb Z$, and the algebras to be smooth. We keep the latter assumption as otherwise already the de~Rham complex is problematic; of course, one may wonder about ``derived $q$-de~Rham cohomology'' in the singular case.

Note that for the definition of the $q$-de~Rham complexes, the assumption $A=\bb Z$ was unnecessary: Let $A$ be any ring, and $R$ a smooth $A$-algebra, with an \'etale map $\square: A[T_1,\ldots,T_d]\to R$. Then one can define $q\!\op-\!\Omega^\bullet_{R[[q-1]]/A[[q-1]],\square}$ exactly as before. Let us discuss whether this should be independent of $\square$ up to quasi-isomorphism in several examples.

\begin{example}\label{ex:variants}
\begin{altenumerate}
\item[{\rm (i)}] The algebra $A$ is \'etale over $\bb Z$. In this case, the complex should be independent of $\square$. In fact,
\[
q\!\op-\!\Omega^\bullet_{R[[q-1]]/A[[q-1]],\square} = q\!\op-\!\Omega^\bullet_{R[[q-1]]/\bb Z[[q-1]],\square^\prime}\ ,
\]
where $\square^\prime: \bb Z[T_1,\ldots,T_d]\to A[T_1,\ldots,T_d]\to R$ is the induced map, which is still \'etale. As the right side should be independent of $\square^\prime$ (in the derived category), the left side should be so, too.
\item[{\rm (ii)}] The algebra $A$ is smooth over $\bb Z$. In this case, the complex should be independent of $\square$ (at least if $\Spec A$ is small enough), but there should be different ways of making it independent of $\square$. More precisely, Conjecture~\ref{conjsmooth} says that whenever one fixes a framing $\square_A: \bb Z[T_1,\ldots,T_d]\to A$, there is a canonical complex $q\!\op-\!\Omega_{R/A,\square_A}$, which can in coordinates be computed by $q\!\op-\!\Omega^\bullet_{R[[q-1]]/A[[q-1]],\square}$.\footnote{This may seem confusing, but a similar phenomenon happens for crystalline cohomology of smooth $\bb F_p$-algebras.}
\item[{\rm (iii)}] The ring $A$ is a $\bb Q$-algebra. In this case, the $q$-de~Rham complex is independent of $\square$, as it is quasi-isomorphic to the constant extension of the de~Rham complex. Thus, as indicated above, the theory is only interesting over mixed characteristic rings.
\item[{\rm (iv)}] The ring $A$ is $\bb F_p$. In this case, we do not expect that the complex is canonically independent of $\square$. Note that any smooth $\bb F_p$-algebra can be lifted to a smooth $\bb Z_p$-algebra, and any two framings can be lifted as well. The two resulting $q$-de Rham complexes over $\bb Z_p$ should be canonically quasi-isomorphic; reduction modulo $p$ produces a quasi-isomorphism between the two $q$-de~Rham complexes over $\bb F_p$. However, we expect that this quasi-isomorphism depends on the choices made. Namely, recall that $q$-de~Rham cohomology becomes isomorphic to singular cohomology over $\bb Z((q-1))$; in particular, over $\bb F_p((q-1))$, it becomes isomorphic to \'etale cohomology with $\bb F_p$-coefficients. If the $q$-de~Rham complex exists canonically for $\bb F_p$-algebras (in a way compatible with the theory over $\bb Z$), this base change to $\bb F_p((q-1))$ would depend only on the fiber of $X$ over $\bb F_p$; however, one can show by examples that \'etale cohomology with $\bb F_p$-coefficients does not depend only on the fiber of $X$ over $\bb F_p$.
\item[{\rm (v)}] Let $A$ be the ring of integers in a finite extension $K$ of $\bb Q_p$. If the extension is unramified, we expect independence; this is basically a special case of (i). In general, we do not expect independence. It is an interesting question how the $q$-de~Rham complex can be extended to places of bad reduction. We note that the theory of \cite{BMS} actually works in the ramified setting. In abstract $p$-adic Hodge theory, the $q$-de~Rham complex is closely related to the theory of Wach modules, \cite{Wach}, \cite{BergerWach}, which only works for abelian extensions of $\bb Q_p$. For general extensions of $\bb Q_p$, there is the theory of Breuil--Kisin modules, \cite{breuil}, \cite{kisin}, which depends on a choice of uniformizer of $K$. One might expect that there is a Breuil--Kisin variant of the $q$-de~Rham complex; however, this seems harder to write down explicitly, as the variable $q$ is tied with the roots of unity. One may also wonder about Lubin--Tate variants.
\end{altenumerate}
\end{example}

\bibliographystyle{acm}
\bibliography{Toulouse}

\end{document}